\documentclass[11pt,reqno]{amsart}%
\usepackage{amssymb,amsmath,amsthm,amsfonts,graphicx}
\usepackage{amsmath}
\usepackage{amsfonts}
\usepackage{amssymb}
\usepackage{graphicx}%
\providecommand{\U}[1]{\protect\rule{.1in}{.1in}}
\newtheorem{thm}{Theorem}[section]
\newtheorem{defn}{Definition}[section]
\newtheorem{lemma}{Lemma}[section]

\newtheorem{cor}{Corollary}[section]
\newtheorem{rmk}{Remark}[section]
\numberwithin{equation}{section}
\begin{document}
\title{maximal hypersurfaces over exterior domains}
\author{Guanghao Hong}
\address{School of Mathematics and Statistics, Xi'an Jiaotong University, Xi'an,
P.R.China 710049.}
\email{ghhongmath@xjtu.edu.cn}
\author{Yu Yuan}
\address{Department of Mathematics, University of Washington, Seattle, WA 98195, USA.}
\email{yuan@math.washington.edu}

\begin{abstract}
In this paper, we study the exterior problem for the maximal surface equation.
We obtain the precise asymptotic behavior of the exterior solution at
infinity. And we prove that the exterior Dirichlet problem is uniquely
solvable given admissible boundary data and prescribed asymptotic behavior at infinity.

\end{abstract}
\date{\today}
\maketitle



\section{Introduction}

The maximal surface equation is
\begin{equation}
div(\frac{Du}{\sqrt{1-|Du|^{2}}})=0, \label{Ediv}%
\end{equation}
or equivalently in the non-divergence form
\begin{equation}
\triangle u+\frac{(Du)^{^{\prime}}D^{2}uDu}{1-|Du|^{2}}=0. \label{Endiv}%
\end{equation}
This equation arises as the Euler equation of the variational problem that
maximize the area functional $\int\sqrt{1-|Du|^{2}}$ among the spacelike
hypersurfaces in the Lorentz-Minkowski space $\mathbb{L}^{n+1}$ (see the
definitions in section 2). The graph of a solution to (\ref{Ediv}) is called a
maximal hypersurface and the graph of a solution to the variational problem is
called an area maximizing hypersurface.

Calabi [Ca68] ($n\leq4$) and Cheng-Yau [CY76] (all dimensions) proved that
every entire maximal hypersurface in $\mathbb{L}^{n+1}$ or every global
solution $u$ to the maximal surface equation (\ref{Ediv}) with $\left\vert
Du\left(  x\right)  \right\vert <1$ on $\mathbb{R}^{n}$ must be linear.

The Dirichlet problem for bounded domain was studied by Bartnik-Simon [BS82]
and the isolated singularity problem was studied by Ecker [Ec86]. The exterior
problem is a \textquotedblleft complimentary\textquotedblright\ one for
elliptic equations; see for example [Be51][Si87] for minimal hypersurfaces,
[CL03] for Monge-Ampere equation, [LLY17] for special Lagrangian equation and
[HZ18] for infinity harmonic functions, besides the classic works such as
[GS56] for linear ones. We study the exterior problem for the maximal surface
equation in this paper. We obtain the precise asymptotic behavior of the
exterior solution at infinity. And we prove that the exterior Dirichlet
problem is uniquely solvable.

Throughout the paper, we assume $A\subset\mathbb{R}^{n}$ be a bounded closed
set. We say $u$ is an exterior solution in $\mathbb{R}^{n}\backslash A$ if
$u\in C^{2}(\mathbb{R}^{n}\backslash A)$ with $|Du(x)|<1$ solve the equation
(\ref{Ediv}) in $\mathbb{R}^{n}\backslash A$. Given an exterior solution $u,$
for any bounded $C^{1}$ domain $U\supset A,$ the integral $Res[u]:=\int%
_{\partial U}\frac{\partial u/\partial\vec{n}}{\sqrt{1-|Du|^{2}}}d\sigma$ is
independent of the choices of $U$ because of the divergence structure of the
equation. The number $Res[u]$ can be regarded as the residue of the exterior
solution $u.$

\begin{thm}
Let $u$ be a smooth exterior solution in $\mathbb{R}^{n}\backslash A$ with
$A\ $being bounded. Then there exist a vector $a\in\mathring{B}_{1}$ and a
constant $c\in\mathbb{R}$ such that for $n=2$
\begin{align}
u(x)  &  =a\cdot x+(1-|a|^{2})Res[u]\ln\sqrt{|x|^{2}-(a\cdot x)^{2}%
}+c\nonumber\\
&  +Res[u]|a|\sqrt{1-|a|^{2}}\frac{|x|(a\cdot x)}{|x|^{2}-(a\cdot x)^{2}}%
\cdot\frac{\ln|x|}{|x|}+O_{k}(|x|^{-1}) \label{Asymptote2d}%
\end{align}
and for $n\geq3$
\begin{equation}
u(x)=a\cdot x+c-(1-|a|^{2})Res[u](\sqrt{|x|^{2}-(a\cdot x)^{2}})^{2-n}%
+O_{k}(|x|^{1-n}) \label{Asymptote3d}%
\end{equation}
as $|x|\rightarrow\infty$ for all $k=0,1,2,\cdots$. The notation
$\varphi(x)=O_{k}(|x|^{m})$ means that $|D^{k}\varphi(x)|=O(|x|^{m-k})$.
\end{thm}

On the other hand, for any bounded closed set $A$, given an admissible
boundary value function $g:\partial A\rightarrow\mathbb{R}$ and prescribed
asymptotic behavior at infinity, the exterior Dirichlet problem for maximal
surface equation is uniquely solvable. We say $g$ is admissible if $g$ is
bounded and there exists a spacelike function $\psi$ in $\mathbb{R}%
^{n}\backslash A$ such that $\psi=g$ on $\partial A$ in the sense of (1.1) in
[BS82] (see Remark 2.1 in Section 2).

\begin{thm}
Let $A\subset\mathbb{R}^{n}$ be a bounded closed set and $g:\partial A
\rightarrow\mathbb{R}$ be an admissible boundary value function. Then

\begin{enumerate}
\item $n=2$, given any $a\in B_{1}$ and $d\in\mathbb{R}$, there exist a unique
smooth solution $u$ of maximal surface equation on $\mathbb{R}^{2}\backslash
A$ such that $u=g$ on $\partial A$ and
\[
u(x)=a\cdot x+d\ln\sqrt{|x|^{2}-(a\cdot x)^{2}}%
+O(1)\ \ \mbox{as}\ x\rightarrow\infty;
\]

\item $n\geq3$, given any $a\in B_{1}$ and $c\in\mathbb{R}$, there exist a
unique smooth solution $u$ of maximal surface equation on $\mathbb{R}%
^{n}\backslash A$ such that $u=g$ on $\partial A$ and
\[
u(x)=a\cdot x+c+o(1)\ \ \mbox{as}\ x\rightarrow\infty.
\]

\end{enumerate}

Of course $u$ enjoys finer asymptotic properties and the relation
$d=(1-|a|^{2})Res[u]$ holds by Theorem 1.1.
\end{thm}

The article is organized as follows. In Section 2, we set up some notations
and definitions, and we collect some results from [CY76], [BS82], and [Ec86]
that are needed in the proofs of the later sections. In Section 3, we prove
that a spacelike function over an exterior domain can be spacelikely extended
to the whole $\mathbb{R}^{n}$. This is the starting point of our work.
Interestingly there is a striking similarity between our argument and the
argument in [CL03, p.571-572] where Caffarelli and Li prove the locally convex
solution of $\det D^{2}u=1$ over an exterior domain can be extended (after
finitely enlarging the complimentary domain $A$) to a global convex function.
In Section 4, we prove a growth control theorem for the exterior solution $u$
at infinity. This is the key content of this paper. Inspired by Ecker's proof
in [Ec86], and relying on his results there, our argument involves
compactness, blowdown analysis and comparison principle. In Section 5, we
prove gradient estimate for $u$ based on the growth control theorem and
Cheng-Yau's estimate on the second fundamental form. In Section 6, we prove
Theorem 1.1. Since the equation (\ref{Ediv}) becomes uniformly elliptic by the
gradient estimate of the previous section, the standard tools such as Harnack
inequality and Schauder estimate apply. The known radially symmetric solutions
play a key role in the proof. In Section 7, we prove Theorem 1.2. We solve the
equation in a series of bigger and bigger ring-shaped domains and use the
compactness method to get an exterior solution. We use the Lorentz
transformations of radially symmetric solutions as barrier functions to locate
the position of the exterior solution. The uniqueness of solutions follows
from comparison principle.

\section{Notations and preliminary results}

We denote the Lorentz-Minkowski space by $\mathbb{L}^{n+1}=\{X=(x,t):x\in
\mathbb{R}^{n},t\in\mathbb{R}\}$, with the flat metric $\sum_{i=1}^{n}%
dx_{i}^{2}-dt^{2}$. And $\langle\cdot,\cdot\rangle$ denotes the inner product
in $\mathbb{L}^{n+1}$ with the signature $(+,\cdots,+,-)$.

The light cone at $X_{0}=(x_{0},t_{0})\in\mathbb{L}^{n+1}$ is defined by
\[
C_{X_{0}}=\{X\in\mathbb{L}^{n+1}:\langle X-X_{0},X-X_{0}\rangle=0\}.
\]
The upper and lower light cones will be denoted by $C^{+}_{X_{0}}$ and
$C^{-}_{X_{0}}$ respectively.

The Lorentz-balls are defined by
\[
L_{R}(X_{0})=\{X\in\mathbb{L}^{n+1}:\langle X-X_{0},X-X_{0}\rangle< R^{2}\}.
\]

Let $M$ be an $n$-dimensional hypersurface in $\mathbb{L}^{n+1}$ which can be
represented as the graph of $u\in C^{0,1}(\Omega)$, where $\Omega$ is a open
set in $\mathbb{R}^{n}$. We say that $M$ (or $u$) is

\textit{weakly spacelike} if $|Du|\leq1$ a.e. in $\Omega$,

\textit{spacelike} if $|u(x)-u(y)|<|x-y|$ whenever $x,y\in\Omega$, $x\neq y$
and the line segment $\overline{xy} \subset\Omega$, and

\textit{strictly spacelike} if $u\in C^{1}(\Omega)$ and $|Du|<1$ in $\Omega$.

If $M$ (or $u$) is strictly spacelike and $u\in C^{2}(\Omega)$, the Lorentz
metric on $\mathbb{L}^{n+1}$ induces a Riemannian metric $g$ on $M$. Under the
coordinates $(x_{1},\cdots,x_{n})\in\Omega$, $g_{ij}=\langle\frac{\partial
X}{\partial x_{i}},\frac{\partial X}{\partial x_{j}}\rangle=\delta_{ij}%
-u_{i}u_{j}$, where $X=(x,u(x))$ is the position vector on the graph of $u$,
and $u_{k}=u_{x_{k}}=\frac{\partial u}{\partial x_{k}}$ for $k=1,\cdots,n$. So
$g=I-Du(Du)^{^{\prime}}$, $\det g=1-|Du|^{2}$, $g^{-1}=I+\frac
{Du(Du)^{^{\prime}}}{1-|Du|^{2}}$ and $g^{ij}=\delta_{ij}+\frac{u_{i}u_{j}%
}{1-|Du|^{2}}$. The second fundamental form is $II_{ij}=\frac{u_{ij}}%
{\sqrt{\det g}}$ and so $|II|^{2}=\frac{g^{ij}g^{kl}u_{ik}u_{jl}}{\det g}$
(see (2.3) in [BS82]) where $u_{ij}=\frac{\partial^{2}u}{\partial
x_{i}\partial x_{j}}$ and the summation convention on repeated indices is
used. Note that $|D^{2}u|\leq|II|.$

The following fundamental results were achieved by Bartnik and Simon in [BS82].

\begin{thm}
[{Solvability of variational problem on bounded domains [BS82, Proposition
1.1]}]Let $\Omega\subset\mathbb{R}^{n}$ be a bounded domain and let
$\varphi:\partial\Omega\rightarrow R$ be a bounded function. Then the
variational problem
\begin{equation}
\sup_{v\in K}\int_{\Omega}\sqrt{1-|Dv|^{2}} \label{VarProb}%
\end{equation}
where $K=\{v\in C^{0,1}(\Omega):|Dv|\leq1$ a.e. in $\Omega$, $v=\varphi$ on
$\partial\Omega\}$ has a unique solution $u$ if and only if the set $K$ is nonempty.
\end{thm}

\begin{rmk}
In above theorem, $v=\varphi$ on $\partial\Omega$ means that, for every
$x_{0}\in\partial\Omega$ and every open straight line segment $l$ contained in
$\Omega$ and with endpoint $x_{0}$,
\[
\lim_{x\rightarrow x_{0},x\in l}v(x)=\varphi(x_{0}).
\]
Regarding this definition and the existence of weakly spacelike extension of
$\varphi$, we refer the readers to the discussion in [BS82, p.133, p.148--149].
\end{rmk}

\begin{defn}
[Area maximizing hypersurface]A weakly spacelike function $u\in C(\Omega)$
($\Omega\subset\mathbb{R}^{n}$ is not necessarily bounded) is called
\textit{area maximizing} if it solves the variational problem (\ref{VarProb})
with respect to its own boundary values for every bounded subdomain in
$\Omega$. The graph of $u$ is called an \textit{area maximizing hypersurface}.
\end{defn}

\begin{lemma}
[{Closeness of variational solutions [BS82, Lemma 1.3]}]If $\{u_{k}\}$ is a
sequence of area maximizing functions in $\Omega$ and $u_{k}\rightarrow u$ in
$\Omega$ locally uniformly, then $u$ is also an area maximizing function.
\end{lemma}

One key result in [BS82, Theorem 3.2] is that if an area maximizing
hypersurface contains a segment of light ray, then it contains the whole of
the ray extended all the way to the boundary or to infinity. This implies the
following conclusion.

\begin{thm}
[The relationship between the variational solutions and the solutions of
maximal surface equation]The solution $u$ of (\ref{VarProb}) is smooth and
solves equation (\ref{Ediv}) in
\[
reg\ u:=\Omega\backslash sing\ u
\]
where
\[
sing\ u:=\{\overline{xy}:x,y\in\partial\Omega,x\neq y,\overline{xy}%
\subset\Omega\ \mbox{and}\ |\varphi(x)-\varphi(y)|=|x-y|\}.
\]
Furthermore
\[
u(tx+(1-t)y)=t\varphi(x)+(1-t)\varphi(y),\ 0<t<1
\]
where $x,y\in\partial\Omega$ are such that $\overline{xy}\subset\Omega$ and
$|\varphi(x)-\varphi(y)|=|x-y|$.
\end{thm}

\begin{rmk}
[Solvability of maximal surface equation on bounded domains]If the boundary
data $\varphi$ admits a weakly spacelike extension and satisfies that
$|\varphi(x)-\varphi(y)|<|x-y|$ for all $x,y\in\partial\Omega$ with
$\overline{xy}\subset\Omega$ and $x\neq y$, then $sing\ u=\emptyset$ and hence
smooth $u$ solves the equation (\ref{Ediv}) in $\Omega$.
\end{rmk}

Bartnik proved the following

\begin{thm}
[{Bernstein theorem for variational solutions [Ec86, Theorem F]}]Entire area
maximizing hypersurfaces in $L^{n+1}$ are weakly spacelike hyperplanes.
\end{thm}

\begin{defn}
[{Isolated singularity [Ec86, p.382]}]A weakly spacelike hypersurface $M$ in
$L^{n+1}$ containing $0$ is called an area maximizing hypersurface having an
isolated singularity at $0$ if $M\backslash\{0\}$ is area maximizing but $M$
cannot be extended as an area maximizing hypersurface into $0$.
\end{defn}

For a weakly spacelike entire or exterior hypersurface $M$ (\textit{i.e.}, $u$
is defined on $\mathbb{R}^{n}$ or an exterior domain $\mathbb{R}^{n}\backslash
A$ with $A$ bounded), we define $M_{r}=r^{-1}M$ with $r>0$ is the graph of
$u_{r}(x)=r^{-1}(rx)$. If for some $r_{j}\rightarrow+\infty$, $u_{r_{j}}(x)$
converge locally uniformly to a function $u_{\infty}(x)$ on $\mathbb{R}^{n}$
or $\mathbb{R}^{n}\backslash\{0\}$, then $u_{\infty}$ (its graph $M_{\infty}$)
is called a blowdown of $u$ ($M$). Note that by weakly spacelikeness,
Arzela-Ascoli theorem always ensures the existence of blowdowns. By Lemma 2.1,
$u_{\infty}(x)$ ($M_{\infty}$) is area maximizing on $\mathbb{R}^{n}$ or
$\mathbb{R}^{n}\backslash\{0\}$ and $u_{\infty}(0)=0$.

Ecker proved that the isolated singularities of area maximizing hypersurface
are light cone like ([Ec86, Theorem 1.5]). The following lemma will also be
used in our proof of Theorem 1.1.

\begin{lemma}
[{[Ec86, Lemma 1.10]}]Let $M$ be an entire area maximizing hypersurface having
an isolated sigularity at $0$ and assume that some blowdown of $M$ also has an
isolated singularity at $0$. Then $M$ has to be either $C_{0}^{+}$ or
$C_{0}^{-}$.
\end{lemma}

We also need the following radial, catenoid like solutions to the maximal
surface equation of (\ref{Ediv}) in $\mathbb{R}^{n}\backslash\{0\},$ used as
barriers in [BS82] and [Ec86]. For $\lambda\in\mathbb{R},$ set
\begin{equation}
w_{\lambda}(x):=\int_{0}^{|x|}\frac{\lambda}{\sqrt{t^{2(n-1)}+\lambda^{2}}}dt.
\label{SolRadial}%
\end{equation}
For $n\geq2$, the integral $\int_{0}^{+\infty}\frac{\lambda}{\sqrt
{t^{2(n-1)}+\lambda^{2}}}dt$ is bounded and we denote this value as
$M(\lambda,n)$. More precisely, by computation
\begin{equation}
\int_{0}^{r}\frac{\lambda}{\sqrt{t^{2(n-1)}+\lambda^{2}}}dt=M(\lambda
,n)-\frac{\lambda}{n-2}r^{2-n}+O(r^{4-3n}) \label{SolRadialAsymp3d}%
\end{equation}
for large $r$. It is obvious that $M(\lambda,n)=sign(\lambda)|\lambda
|^{\frac{1}{n-1}}M(1,n)\rightarrow\pm\infty$ as $\lambda\rightarrow\pm\infty$
and $M(\lambda,n)\rightarrow0$ as $\lambda\rightarrow0$. For $n=2$, the
integral $\int_{0}^{+\infty}\frac{\lambda}{\sqrt{t^{2}+\lambda^{2}}}dt$ is
infinite and by computation
\begin{equation}
\int_{0}^{r}\frac{\lambda}{\sqrt{t^{2}+\lambda^{2}}}dt=m(\lambda)+\lambda\ln
r+O(r^{-2}) \label{SolRadialAsymp2d}%
\end{equation}
for large $r$, where $m(\lambda)=\int_{0}^{1}\frac{\lambda}{\sqrt
{t^{2}+\lambda^{2}}}dt+\int_{1}^{+\infty}(\frac{\lambda}{\sqrt{t^{2}%
+\lambda^{2}}}-\frac{\lambda}{t})dt$.

\begin{defn}
[Lorentz transformations, the speed of light is normalized to 1]For a
parameter $\kappa\in(-1,1)$, the Lorentz transformation $L_{\kappa}:
\mathbb{L}^{n+1}\rightarrow\mathbb{L}^{n+1}$ is defined as
\[
L_{\kappa}: (x^{\prime},x_{n},t)\rightarrow(x^{\prime},\frac{x_{n}+\kappa
t}{\sqrt{1-\kappa^{2}}},\frac{\kappa x_{n}+ t}{\sqrt{1-\kappa^{2}}})
\]
where $x^{\prime}=(x_{1},\cdots,x_{n-1})$.
\end{defn}

The Lorentz transformations are isometries of $\mathbb{L}^{n+1}$. $L_{\kappa}$
maps spacelike (weakly spacelike) surfaces to spacelike (weakly spacelike)
surfaces and it maps maximal surfaces (area maximizing surfaces) to maximal
surfaces (area maximizing surfaces). Geometrically $L_{\kappa}$ can be seen as
a hyperbolic rotation. It maps the light cone $\{(x,t)\in\mathbb{L}^{n+1}:
t=|x|\}$ to itself and it maps the horizontal hyperplanes to the hyperplanes
with slope $\kappa$:
\[
L_{\kappa}(\{(x,t)\in\mathbb{L}^{n+1}: t=T\})=\{(x,t)\in\mathbb{L}^{n+1}:
t=\sqrt{1-\kappa^{2}}T+\kappa x_{n}\},
\]
for $T\in(-\infty, +\infty)$.

More generally, for any vector $a\in B_{1}$ we define $L_{a}:=T_{a}%
L_{|a|}T_{a}^{-1}$ where $T_{a}$ is a rotation that keeps $t$-axis fixed and
transforms $e_{n}$ to $\frac{a}{|a|}$ in $\mathbb{R}^{n}$ (in case of $a=0$ we
just define $T_{0}:=id$).

\section{Extension of spacelike hypersurface with hole}

We start our proofs for the two main theorems by extending any spacelike
function over an exterior domain to a global spacelike function, after
finitely enlarging the bounded complimentary domain.

\begin{thm}
Let $u$ be a spacelike function in $\mathbb{R}^{n}\backslash A$ with $A$ being
bounded. Then there exists $R^{\ast}>0$ such that $\left\vert u\left(
x\right)  -u\left(  y\right)  \right\vert <|x-y|$ for all $x,y\in$
$\mathbb{R}^{n}\backslash B_{R^{\ast}}.$
\end{thm}

\begin{proof}
\textit{Step 1.} We first show that there exists a ball $B_{R_{0}}%
(x_{0})\supset A$ such that on the boundary $osc_{x\in\partial B_{R_{0}}%
(x_{0})}u(x)<2R_{0}.$ Without loss of generality we assume $A\subset B_{1}$.
We suppose $osc_{\partial B_{100}}u(x)\geq200$ with $\max_{\partial B_{100}%
}u(x)=u(100e_{1})$ and we will show that $osc_{\partial B_{200}(100e_{2}%
)}u(x)<400$.

Suppose $\max_{\partial B_{100}}u(x)=-\min_{\partial B_{100}}u(x)$ because
otherwise we can consider $u-(\max_{\partial B_{100}}u+\min_{\partial B_{100}%
}u)/2$ in place of $u$. Firstly one can see that $osc_{\partial B_{100}%
}u(x)\leq202$ from the Lipschitz condition on $u$ and the geometry of $\bar
{B}_{100}\backslash B_{1}$. So $100\leq u(100e_{1})\leq101$ and $\min
_{\partial B_{100}}u(x)\in(-101,-100)$. Suppose $u(x_{1})=\min_{\partial
B_{100}}u(x)$ for some $x_{1}\in\partial B_{100}$. Then $|x_{1}-(-100e_{1}%
)|\leq3$ because $u(x)>u(100e_{1})-|100e_{1}-x|>100-200=-100$ for any
$x\in\partial B_{100}\backslash B_{3}(-100e_{1})$. Thus $u(-100e_{1}%
)\in(-104,-97)$. Therefore $u(100e_{2})>u(100e_{1})-|100e_{1}-100e_{2}%
|\geq100-100\sqrt{2}>-42$ and $u(100e_{2})<u(-100e_{1})+|-100e_{1}%
-100e_{2}|<-97+100\sqrt{2}<45$. In the same way, $u(-100e_{2})\in(-42,45)$.
Denote $u(100e_{2})=M$. Then $u(x)\in(M-90,M+90)$ for all $x\in B_{3}%
(-100e_{2})$. Let $\max_{x\in\partial B_{200}(100e_{2})\backslash
B_{3}(-100e_{2})}|u(x)-M|:=Q<200$. Therefore $osc_{\partial B_{200}(100e_{2}%
)}u(x)\leq2\max(Q,90)<400$. (See Figure 1.)\begin{figure}[ptb]
\centering
\includegraphics[width=0.9\linewidth]{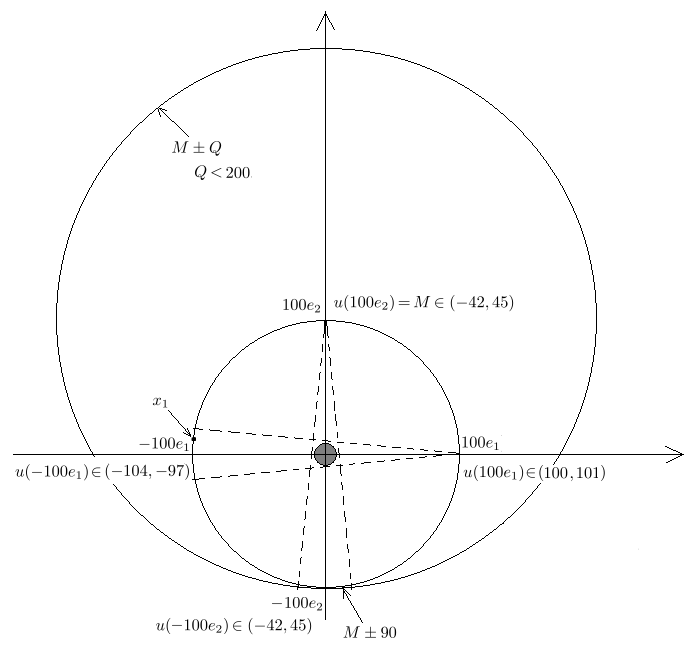}\newline\caption{Shift the ball to
hide the shadow}%
\end{figure}

\textit{Step 2.} We show that there exists $R_{1}>R_{0}$ such that for all
$R\geq R_{1}$ we have $|u(x)-u(y)|<|x-y|$ for all $x,y\in\partial B_{R}%
(x_{0})$ with $x\neq y.$ By making a suitable transformation, we may assume
$x_{0}=0$, $R_{0}=1$ and $\max_{\partial B_{1}}u=-\min_{\partial B_{1}%
}u=1-\epsilon_{0}$ for some $0<\epsilon_{0}<1$. Then for $R>1$, $\max
_{\partial B_{R}}|u|\leq R-\epsilon_{0}$. For $x,y\in\partial B_{R}$ with
$x\neq y$, if the line segment $\overline{xy}\subset\bar{B}_{R}\backslash
B_{1}$ then $|u(x)-u(y)|<|x-y|$. Otherwise, $dist(0,\overline{xy})<1$ and
$|x-y|>2\sqrt{R^{2}-1}$. If $R\geq\frac{1+\epsilon_{0}^{2}}{2\epsilon_{0}}$
then $|u(x)-u(y)|\leq2(R-\epsilon_{0})\leq2\sqrt{R^{2}-1}<|x-y|.$

\textit{Step 3.} Set $R^{\ast}:=|x_{0}|+R_{1}.$ Suppose the line segment
$\overline{xy}\cap\partial B_{R_{1}}(x_{0})=\{p,q\}$ and $p$ is closer to $x$
than $q$. Then $|u(x)-u(y)|\leq
|u(x)-u(p)|+|u(p)-u(q)|+|u(q)-u(y)|<|x-p|+|p-q|+|q-y|=|x-y|$. If $p=q$, the
conclusion is also true. If $\overline{xy}\cap\partial B_{R_{1}}%
(x_{0})=\emptyset,$ we have $|u(x)-u(y)|<|x-y|$ directly.
\end{proof}

For completeness, we include the promised full spacelike extension result
here, which is not needed in the proof of our two main theorems.

\begin{thm}
Let $u$ be a spacelike function in $\mathbb{R}^{n}\backslash A$ with $A$ being
bounded. Then there exists $R^{\ast}>0$ such that $\left\vert u\left(
x\right)  -u\left(  y\right)  \right\vert <|x-y|$ for all $x,y\in$
$\mathbb{R}^{n}\backslash B_{R^{\ast}}.$ Moreover, there exists a spacelike
function $\tilde{u}$ in $\mathbb{R}^{n}$ such that $\tilde{u}=u$ in
$\mathbb{R}^{n}\backslash B_{R^{\ast}}.$
\end{thm}

\begin{proof}
We only need to prove the second part of the theorem. By Remark 2.2, there
exists a spacelike function $w$ in $B_{R^{\ast}}$ such that $w=u$ on $\partial
B_{R^{\ast}}.$ Define $\tilde{u}:=w$ in $B_{R^{\ast}}$ and $\tilde{u}:=u$ in
$\mathbb{R}^{n}\backslash B_{R^{\ast}}$. For $x,y\in\mathbb{R}^{n}$ with
$x\neq y,$ if both $x$ and $y$ are in $\bar{B}_{R^{\ast}}$ or $\mathbb{R}%
^{n}\backslash B_{R^{\ast}}$ then $|u(x)-u(y)|<|x-y|$. Otherwise, let
$\{z\}=\overline{xy}\cap\partial B_{R^{\ast}}$, then $|u(x)-u(y)|\leq
|u(x)-u(z)|+|u(z)-u(y)|<|x-z|+|z-y|=|x-y|.$

If we assume the spacelike function $u$ is also strictly spacelike $\left\vert
Du\left(  x\right)  \right\vert <1$ (to exclude spacelike functions such as
$\arctan\lambda$), we can get a spacelike extension inside $B_{R^{\ast}}$
directly, without relying on the singularity analysis of variational solutions
to the maximal surface equations of [BS82] contained in Remark 2.2.

In fact (cf. [LY, p.61]), for $x\in\bar{B}_{R^{\ast}}$ set%
\[
w\left(  x\right)  =\inf_{b\in\partial B_{R^{\ast}}}\left\{  u\left(
b\right)  +m\left\vert x-b\right\vert \right\}
\]
with $m=\left\Vert Du\right\Vert _{L^{\infty}\left(  \partial B_{R^{\ast}%
}\right)  }<1.$ Then $w\left(  x\right)  =u\left(  x\right)  $ for
$x\in\partial B_{R^{\ast}}.$ And for $x,y\in\bar{B}_{R^{\ast}}$%
\begin{align*}
w\left(  y\right)   &  =\inf_{b\in\partial B_{R^{\ast}}}\left\{  u\left(
b\right)  +m\left\vert y-b\right\vert \right\} \\
&  \leq\inf_{b\in\partial B_{R^{\ast}}}\left\{  u\left(  b\right)
+m\left\vert x-b\right\vert +m\left\vert y-x\right\vert \right\} \\
&  \leq w\left(  x\right)  +m\left\vert y-x\right\vert .
\end{align*}
Symmetrically $w\left(  x\right)  \leq w\left(  y\right)  +m\left\vert
x-y\right\vert .$ Hence $w$ is spacelike inside $B_{R^{\ast}},$ $\left\vert
u\left(  x\right)  -u\left(  y\right)  \right\vert <m\left\vert x-y\right\vert
<\left\vert x-y\right\vert .$

There is another differential way to do this extension inside $B_{R^{\ast}}.$
Without loss of generality, we assume $R^{\ast}=1,$ then $osc_{\left\vert
x\right\vert =1}u(x)<2.$ For $x\in\bar{B}_{1}$ set%
\[
w\left(  x\right)  =\left\vert x\right\vert \left[  u\left(  x/\left\vert
x\right\vert \right)  -m\right]  +m
\]
with $m=\frac{1}{2}\left[  \max_{\left\vert x\right\vert =1}u\left(  x\right)
+\min_{\left\vert x\right\vert =1}u\left(  x\right)  \right]  .$ Then
$w\left(  x\right)  =u\left(  x\right)  $ on $\partial B_{1}.$ And for $x\in
B_{1}\backslash\left\{  0\right\}  ,$%
\[
\left\vert Dw\left(  x\right)  \right\vert =\left\vert Du\left(  x/\left\vert
x\right\vert \right)  \right\vert <1.
\]
The Lipschitz norm of $w$ at $x=0$ is also less than one because%
\[
\left\vert u\left(  x/\left\vert x\right\vert \right)  -m\right\vert \leq
\frac{\max_{\left\vert x\right\vert =1}u\left(  x\right)  -\min_{\left\vert
x\right\vert =1}u\left(  x\right)  }{2}<1.
\]
We also reach the same spacelike conclusion of $w$ inside $B_{1}.$
\end{proof}

\section{Growth control of $u$ at infinity}

In this section, we show that the linear growth rate of an exterior solution
$u$ at infinity is uniformly less than one, that is to say, $u$ is controlled
not only by the light cone but by a cone with slop less than one. Meanwhile we
prove that the blowdown of $u$ is unique and is a linear function with slope
less than one. We also proved that the graph of $u$ is supported by a
hyperplane either from below or from above.

\begin{thm}
Let $u$ be an exterior solution in $\mathbb{R}^{n}\backslash A$ with
$A\ $being bounded. Then there exist $B_{R}\supset A$, $0<\epsilon<1$ and
$c_{0}\in\mathbb{R}$ such that
\[
-(1-\epsilon)|x|\leq u(x)-c_{0}\leq(1-\epsilon)|x|
\]
in $\mathbb{R}^{n}\backslash B_{R}$. Moreover, there exists a vector $a\in
\bar{B}_{1-\epsilon}$ such that
\[
\lim_{r\rightarrow\infty}\frac{u(rx)}{r}=a\cdot x\quad
\mathrm{locally\ uniformly\ in}\ \mathbb{R}^{n}.
\]
The function $u$ also enjoys the property that either for some $c\in
\mathbb{R}$, $u(x)\geq a\cdot x+c$ in $\mathbb{R}^{n}\backslash B_{R}$ and
$u(y)=a\cdot y+c$ at some point $y\in\partial B_{R}$ or for some
$c\in\mathbb{R}$, $u(x)\leq a\cdot x+c$ in $\mathbb{R}^{n}\backslash B_{R}$
and $u(y)=a\cdot y+c$ at some point $y\in\partial B_{R}$.
\end{thm}

\begin{proof}
We apply Theorem 3.1. For simplicity of notation, we assume $R^{\ast}=1$. So
we have $|u(x)-u(y)|<|x-y|$ for any $x,y\in\mathbb{R}^{n}\backslash B_{1}$
with $x\neq y$. We also assume $\max_{\partial B_{1}}u=-\min_{\partial B_{1}%
}u=1-\epsilon_{1}$ for some $0<\epsilon_{1}<1$. We will show that
$-(1-\epsilon)|x|\leq u(x)\leq(1-\epsilon)|x|$ in $\mathbb{R}^{n}\backslash
B_{1}$ for some $0<\epsilon<1$.

It is easy to see that $-|x|+\epsilon_{1}\leq u(x) \leq|x|-\epsilon_{1}$ in
$\mathbb{R}^{n}\backslash B_{1}$. So there are four possibilities for $u$:

\textit{(a)} There is $0<\epsilon<1$ such that $u(x)\geq-(1-\epsilon)|x|$ in
$\mathbb{R}^{n}\backslash B_{1}$ and there is a sequence of points $\{x_{j}\}$
with $1<|x_{j}|:=R_{j}\rightarrow+\infty$ such that $u(x_{j})>(1-\frac{1}%
{j})|x_{j}|$;

\textit{(b)} The function $-u$ satisfies (a);

\textit{(c)} There are two sequences of points $\{x^{\pm}_{j}\}$ with
$1<|x^{\pm}_{j}|:=R^{\pm}_{j}\rightarrow+\infty$ such that $u(x^{+}%
_{j})>(1-\frac{1}{j})|x^{+}_{j}|$ and $u(x^{-}_{j})<-(1-\frac{1}{j})|x^{-}%
_{j}|$;

\textit{(d)} There is $0<\epsilon<1$ such that $-(1-\epsilon)|x|\leq u(x)
\leq(1-\epsilon)|x|$ in $\mathbb{R}^{n}\backslash B_{1}$.

We will show that the cases (a)(b)(c) can not happen.

Suppose that $u$ satisfies (a). Let $\hat{x}:=\lim_{k\rightarrow\infty}%
\frac{x_{j_{k}}}{R_{j_{k}}}\in\partial B_{1}$ for some subsequence $\{j_{k}%
\}$. We assume $\hat{x}=e_{n}$ and consider $\{j_{k}\}$ as $\{j\}$. Define
$v_{j}(x):=\frac{u(R_{j}x)}{R_{j}}$. A subsequence of $v_{j}(x)$ (still
denoted as $v_{j}(x)$) converge locally uniformly to a function $V(x)$ in
$\mathbb{R}^{n}\backslash\{0\}$. By Lemma 2.1, $V(x)$ is \textit{area
maximizing} in $\mathbb{R}^{n}\backslash\{0\}$. It is obvious that $V(0)=0$,
$V(e_{n})=1$ and $V(x)\geq-(1-\epsilon)|x|$. Thus $V(te_{n})=t$ for
$t\in(0,+\infty)$ by weakly spacelikeness and Theorem 2.1. If $0$ is a
removable singularity for $V$, then $V$ is a plane by Theorem 2.3 and $V$ has
to be $V(x)=x_{n}$ that contradicts $V(x)\geq-(1-\epsilon)|x|$. So $0$ is an
isolated singularity for $V$. Let $V_{\infty}$ be a blowdown of $V$, then
$V_{\infty}(te_{n})=t$ for $t\in(0,+\infty)$ and $V_{\infty}(x)\geq
-(1-\epsilon)|x|$. So $0$ is an isolated singularity for $V_{\infty}.$ By
Lemma 2.2, we have $V(x)=|x|$.

Let $z\in\partial B_{1}$ be such that $u(z)=\min_{\partial B_{1}}u:=\lambda$.
For small $\delta>0$, consider $w(x):=\lambda-1+\delta+(1-\delta)|x|$. Since
$\lim_{j\rightarrow\infty}\frac{u(R_{j}x)}{R_{j}}=|x|$ uniformly on $\partial
B_{1}$, for sufficiently large $j$, $u(x)\geq w(x)$ on $\partial B_{R_{j}}$.
But $u(x)\geq\lambda=w(x)$ on $\partial B_{1}$ and $w(x)$ is a subsolution to
(\ref{Ediv}) in $B_{R_{j}}\backslash\bar{B}_{1}$, so $u(x)\geq w(x)$ in
$B_{R_{j}}\backslash\bar{B}_{1}$. Let $\delta\rightarrow0$, we get
$u(x)\geq\lambda-1+|x|$ in $B_{R_{j}}\backslash\bar{B}_{1}$. Especially,
$u(2z)\geq\lambda-1+|2z|=\lambda+1$ and hence $u(2z)-u(z)\geq1=|2z-z|$. This
contradicts the fact that $u$ is spacelike \textquotedblleft in $\mathbb{R}%
^{n}\backslash B_{1}"$ proved in Theorem 3.1.

The case (b) can not happen for the same reason.

Now we suppose $u$ satisfies (c). For each $j$, let $w_{j}$ be the solution of
(\ref{Ediv}) in $B_{j}$ with $w_{j}=u$ on $\partial B_{j}$. The existence of
$w_{j}$ is due to Remark 2.2. For each $j$, either $\max_{\partial B_{1}%
}\left(  w_{j}-u\right)  \geq0$ or $\min_{\partial B_{1}}\left(
w_{j}-u\right)  \leq0$ (or both). Thus $\max_{\partial B_{1}}\left(
w_{j}-u\right)  \geq0$ or $\min_{\partial B_{1}}\left(  w_{j}-u\right)  \leq0$
happens for infinitely many $j$. We assume $\max_{\partial B_{1}}\left(
w_{j}-u\right)  :=\lambda_{j}\geq0$ happens for infinitely many $j$. Let
$z_{j}\in\partial B_{1}$ be such that $w_{j}(z_{j})-u(z_{j})=\lambda_{j}$ and
consider $\tilde{w}_{j}=w_{j}(z_{j})-\lambda_{j}$ for these $j$. So $\tilde
{w}_{j}\leq u$ in $B_{j}\backslash B_{1}$ and $\tilde{w}_{j}(z_{j})=u(z_{j})$.
Note that $|\tilde{w}_{j}(0)|\leq|\tilde{w}_{j}(z_{j})|+1=|u_{j}(z_{j}%
)|+1\leq2$ for all these $j$. Therefore, by Arzela-Ascoli a subsequence
$\tilde{w}_{j_{k}}$ converge locally uniformly to a function $W$ in
$\mathbb{R}^{n}$. By Lemma 2.1, $W$ is an area maximizing surface. So it is a
plane with slope less than or equal to one by Theorem 2.3. Furthermore, we
know $W\leq u$ in $\mathbb{R}^{n}\backslash B_{1}$ and $W(z)=u(z)$ by
continuity, where $z$ is an accumulating point of $\{z_{j_{k}}\}$.

By assumption of (c), there are $\{x_{j}^{-}\}$ with $|x_{j}^{-}%
|\rightarrow+\infty$ such that $W(x_{j}^{-})\leq u(x_{j}^{-})<-(1-\frac{1}%
{j})|x_{j}^{-}|$. Thus $W$ has to be a plane with slope 1. We assume
$DW(x)=e_{n}$, so $W(x)=x_{n}+u(z)-z_{n}$. If $z_{n}<0$, then denote
$\tilde{z}=(z^{\prime},-z_{n})\in\partial B_{1}$ and we have $u(\tilde{z})\geq
W(\tilde{z})=-2z_{n}+u(z)=u(z)+|\tilde{z}-z|$. This contradicts the fact that
$|u(x)-u(y)|<|x-y|$ for any $x,y\in\partial B_{1}$ with $x\neq y$. If
$z_{n}\geq0$, then consider the point $z+e_{n}\in\mathbb{R}^{n}\backslash
\bar{B}_{1}$ and we have $u(z+e_{n})\geq W(z+e_{n})=u(z)+1=u(z)+|(z+e_{n}%
)-z|$. This contradicts the fact that $u$ is spacelike \textquotedblleft in
$\mathbb{R}^{n}\backslash B_{1}"$ proved in Theorem 3.1.

If it is the case that $\min_{\partial B_{1}}\left(  w_{j}-u\right)  \leq0$
happens for infinitely many $j$, we move up $w_{j}$ by $-\min_{\partial B_{1}%
}\left(  w_{j}-u\right)  $ and get a plane $\hat{W}$ above $u$ by the same
process. This time by the assumption that there are $\{x_{j}^{+}\}$ with
$|x_{j}^{+}|\rightarrow+\infty$ such that $\hat{W}(x_{j}^{+})\geq u(x_{j}%
^{+})>(1-\frac{1}{j})|x_{j}^{+}|$, we also know the slope of $\hat{W}$ is one.
Furthermore, $\hat{W}$ also touches $u$ at some point of $\partial B_{1}$.
Again, this contradicts the fact that $u$ is spacelike \textquotedblleft in
$\mathbb{R}^{n}\backslash B_{1}"$ proved in Theorem 3.1.

Therefore only the case (d) can (and must) happens and we have proved the
first part of the theorem. In this case we can also construct the plane $W$ in
the same way just as we did in the first paragraph when we proved the
impossibility of case (c). That is to say, we can place a plane (with slope
less than or equal to $1-\epsilon$) either below or above the graph of $u$ in
$\mathbb{R}^{n}\backslash B_{1}$ and the plane touches $u$ at some point of
$\partial B_{1}$. This property implies that the blowdown of $u$ must be
unique and equal to the blowdown of $W$. We show this as follows. Assume that
$W(x)=c+a\cdot x\leq u$ in $\mathbb{R}^{n}\backslash B_{1}$ where
$|a|\leq1-\epsilon$. Let $V$ be any blowdown of $u$, then $a\cdot x\leq
V(x)\leq(1-\epsilon)|x|$ in $\mathbb{R}^{n}$, which implies that $0$ is a
removable singularity of $V$ by Ecker stated in Lemma 2.2. Then $V$ is an
entire solution and must be a plane. The only possible situation is
$V(x)=a\cdot x$.
\end{proof}

\section{Gradient estimate}

With the strong growth control achieved in the previous section and the known
curvature estimate, we can establish the gradient estimate and ascertain
$Du(\infty)$ in this section. We state the curvature estimate of Cheng-Yau
[CY76] in the following improved extrinsic form carried out by Schoen (see
[Ec86, Theorem 2.2]).

\begin{thm}
Let $M=\left(  x,u\left(  x\right)  \right)  $ be a maximal hypersurface,
$x_{0}\in M$ and assume that for some $\rho>0$, $L_{2\rho}(x_{0})\cap
M\subset\subset M$. Then we have for all $x\in L_{\rho}(x_{0})$
\begin{equation}
|II|^{2}(x)\leq\frac{c(n)\rho^{2}}{(\rho^{2}-l_{x_{0}}^{2}(x))^{2}}
\label{EstCurvature}%
\end{equation}
where $c(n)$ is a constant depending only on the dimension $n$ and $l_{x_{0}%
}(x)=(|x-x_{0}|^{2}-|u(x)-u(x_{0})|^{2})^{\frac{1}{2}}.$
\end{thm}

If $M$ is an entire maximal hypersurface, then $\rho$ in (\ref{EstCurvature})
can be chosen to be arbitrarily large, so $|II|\equiv0$ and hence the
Bernstein Theorem follows. But the following corollary is what we need.

\begin{cor}
For any $0<\epsilon<1$, there exists a positive constant $C(\epsilon,n)$ such
that if $u$ solves the equation (\ref{Ediv}) in $\mathbb{R}^{n}\backslash
\bar{B}_{1}$ and satisfies $-(1-\epsilon)|x|\leq u(x)\leq(1-\epsilon)|x|$ in
$\mathbb{R}^{n}\backslash B_{1}$ then
\[
|II|(x)\leq\frac{C(\epsilon,n)}{|x|}\ \ \ \mbox{for}\ |x|\geq\frac{8}%
{\epsilon}.
\]

\end{cor}

\begin{proof}
Fix a point $x\in\mathbb{R}^{n}$ with $|x|\geq\frac{8}{\epsilon}$, for any
$y\in\partial B_{1}$
\[
|u(x)-u(y)|\leq|u(x)|+|u(y)|\leq(1-\epsilon)|x|+(1-\epsilon)
\]
and
\[
|x-y|\geq|x|-1.
\]
So
\[
l_{x}(y)=(|y-x|^{2}-|u(y)-u(x)|^{2})^{\frac{1}{2}}>\sqrt{\frac{\epsilon}{2}%
}|x|.
\]
This means that
\[
L_{\sqrt{\frac{\epsilon}{2}}|x|}(x)\subset\subset M.
\]
So by letting $x=x_{0}$ and $\rho=\sqrt{\frac{\epsilon}{2}}|x|$ in
(\ref{EstCurvature}) we get
\[
|II|(x)\leq\frac{\sqrt{c(n)}}{\sqrt{\frac{\epsilon}{2}}|x|}=\frac
{C(\epsilon,n)}{|x|}.
\]

\end{proof}

\begin{thm}
Let $u$ be an exterior solution in $\mathbb{R}^{n}\backslash A$. Then for any
open set $U\supset A$ there is $\theta>0$ such that $|Du|\leq1-\theta$ in
$\mathbb{R}^{n}\backslash U$. Moreover, $\lim_{x\rightarrow\infty}Du(x)=a$
where $a$ is given by Theorem 4.1.
\end{thm}

\begin{proof}
Assume $A\subset B_{1}$ and $-(1-\epsilon)|x|\leq u(x)\leq(1-\epsilon)|x|$ in
$\mathbb{R}^{n}\backslash B_{1}$. Denote $\hat{R}:=\frac{10}{\epsilon}$. Since
$|Du(x)|<1$ for $x\in\mathbb{R}^{n}\backslash U$, if $|Du|\leq1-\theta$ is not
true then there is a sequence of points $\{x_{j}\}$ such that $|Du(x_{j}%
)|>1-\frac{1}{j}$ and $|x_{j}|\rightarrow+\infty$. Define $R_{j}:=\hat{R}%
^{-1}|x_{j}|$ (assume $R_{j}>1$) and $v_{j}(x):=\frac{u(R_{j}x)}{R_{j}}$. Then
by Theorem 4.1, we have $v_{j}(x)\rightarrow V(x)=a\cdot x$.

On the other hand, by Corollary 5.1, the curvature $|II|$ is uniformly bounded
for all $v_{j}(x)$ on the compact set $\bar{B}_{\hat{R}+1}\backslash
B_{\hat{R}-1}$, so is $|D^{2}v_{j}(x)|$. This means $Dv_{j}(x)\rightarrow
DV(x)=a$ in $\bar{B}_{\hat{R}+1}\backslash B_{\hat{R}-1}$. Denote
$\lim_{k\rightarrow\infty}\frac{\hat{R}x_{j_{k}}}{|x_{j_{k}}|}=\hat{x}%
\in\partial B_{\hat{R}}$ for some subsequence $j_{k}$. Then $DV(\hat{x}%
)=\lim_{k\rightarrow\infty}Dv_{j_{k}}(\frac{\hat{R}x_{j_{k}}}{|x_{j_{k}}|})$.
But $|Dv_{j_{k}}(\frac{\hat{R}x_{j_{k}}}{|x_{j_{k}}|})|=|Du(x_{j})|>1-\frac
{1}{j_{k}}$ and it implies $|DV(\hat{x})|=1$. This is a contradiction.

The conclusion $\lim_{x\rightarrow\infty}Du(x)=a$ can be proved in the same
compactness way as above.

There is another Harnack way to show the existence of $Du\left(
\infty\right)  ,$ once $\left\vert Du\right\vert $ is uniformly bounded away
from one, $\left\vert Du\right\vert \leq1-\theta.$ Indeed, each bounded
component $u_{k}$ of $Du$ satisfies a uniformly elliptic equation
\[
\partial_{x_{i}}\left[  F_{p_{i}p_{j}}\left(  Du\right)  \partial_{x_{j}}%
u_{k}\right]  =0\ \ \text{in }\mathbb{R}^{n}\backslash A
\]
with $F\left(  p\right)  =\sqrt{1-\left\vert p\right\vert ^{2}}.$ By Moser's
Harnack, in fact [Mo61, Theorem 5], $\lim_{x\rightarrow\infty}u_{k}(x)$ exists.
\end{proof}

Because we will use Moser's results again in next section, we state them here
in the needed form for convenience.

\begin{thm}
[{Harnack inequality [Mo61, Theorem 1]}]Let $w$ be a nonnegative solution of
\begin{equation}
(a_{ij}(x)w_{j})_{i}=0 \label{Euniform}%
\end{equation}
in $\mathbb{R}^{n}\backslash B_{R_{0}},$ where $\Lambda^{-1}I\leq
(a_{ij}(x))\leq\Lambda I$ for for a constant $\Lambda\in\lbrack1,\infty).$
Then for any $R\geq10R_{0}$
\begin{equation}
\sup_{\partial B_{R}}w\leq\Gamma\inf_{\partial B_{R}}w \label{Harnack}%
\end{equation}
for $\Gamma=\Gamma\left(  n,\Lambda\right)  .$
\end{thm}

\begin{thm}
[{Behavior at $\infty$ [Mo61, Theorem 5]}]Let $w$ be a bounded solution to the
uniformly elliptic equation (\ref{Euniform}) in $\mathbb{R}^{n}\backslash
B_{1}$. Then $\lim_{|x|\rightarrow\infty}w(x)$ exists.
\end{thm}

\section{Asymptotic behavior: proof of Theorem 1.1}

Now we are ready to prove Theorem 1.1. We present the proof in the following
four subsections. We first treat the special case $Du\left(  \infty\right)
=a=0$. The general case can be transformed to this special case by a suitable
hyperbolic rotation (Lorentz transformation).

\subsection{Case $a=0$\textbf{, }$n=2$}

\ 

\vspace{0.618ex} \noindent\textit{Step 1.} ($|u(x)|\leq c+d\ln|x|$ for large
$c$ and $d$.)

We still assume $R=1$ in Theorem 4.1. By Theorem 4.1 and Theorem 5.2, we known
that $\lim_{r\rightarrow\infty}\frac{u(rx)}{r}=0$ and $\lim_{x\rightarrow
\infty}|Du(x)|=0$. Moreover, we have either $u(x)\geq c$ for some
$c\in\mathbb{R}$ in $\mathbb{R}^{n}\backslash B_{1}$ and $u(y)=c$ at some
point $y\in\partial B_{1}$ or $u(x)\leq c$ for some $c\in\mathbb{R}$ in
$\mathbb{R}^{n}\backslash B_{1}$ and $u(y)=c$ at some point $y\in\partial
B_{1}$. We assume the former case happens and $c=0$, $y=e_{1}$. That is
$u(x)\geq0$ in $\mathbb{R}^{n}\backslash B_{1}$ and $u(e_{1})=0$. Recall the
radial barrier $w_{\lambda}$ in (\ref{SolRadial}). Set $\phi_{\lambda
}(x):=w_{\lambda}(x)-w_{\lambda}(e_{1})$ and $\psi_{\lambda}(x):=\phi
_{\lambda}(x)+\max_{\partial B_{1}}u$. As the first step of the proof, we want
to show that $u(x)\leq\psi_{\lambda}(x)$ in $\mathbb{R}^{n}\backslash B_{1}$
for sufficiently large $\lambda$.

We observe that as long as $\lambda$ is large enough, $\phi_{\lambda}(2e_{1})$
can be arbitrarily close to 1. Since $u(2e_{1})<1$, we can choose $\lambda
_{0}$ such that $\phi_{\lambda_{0}}(2e_{1})>u(2e_{1})$. Now we claim that
$u(x)\leq\psi_{\lambda_{1}}(x)$ in $\mathbb{R}^{n}\backslash B_{1}$, where
$\lambda_{1}:=(\Gamma+1)\lambda_{0}$ and the constant $\Gamma$ is from Theorem
5.3 for $u$. It is easy to see that $\psi_{\lambda_{1}}(x)>\Gamma\phi
_{\lambda}(x)$ in $\mathbb{R}^{n}\backslash B_{R}$ for some $R=R\left(
\lambda_{0},\Gamma\right)  $ large enough. If $u(x)\leq\psi_{\lambda_{1}}(x)$
in $\mathbb{R}^{n}\backslash B_{R}$ then $u(x)\leq\psi_{\lambda_{1}}(x)$ in
$\mathbb{R}^{n}\backslash B_{1}$ by comparison principle since $u\leq
\max_{\partial B_{1}}u=\psi_{\lambda_{1}}(x)$ on $\partial B_{1}$. Suppose
$u(z)>\psi_{\lambda_{1}}(z)$ at some point $z\in\mathbb{R}^{n}\backslash
B_{R}$, then $u>\phi_{\lambda_{0}}$ on $\partial B_{|z|}$ by Theorem 5.3.
Since $u\geq0=\phi_{\lambda_{0}}$ on $\partial B_{1}$, we have $u\geq
\phi_{\lambda_{0}}$ in $B_{|z|}\backslash B_{1}$ by comparison principle,
especially $u(2e_{1})\geq\phi_{\lambda_{0}}(2e_{1})$. This is a contradiction.
So we proved that $u(x)\leq\psi_{\lambda_{1}}(x)$ in $\mathbb{R}^{n}\backslash
B_{1}$.

\vspace{0.618ex} \noindent\textit{Step 2.} ($u(x)= c+d\ln|x|+o(1)$ for some
$c$ and $d$.)

Denote
\[
\lambda^{\ast}:=\inf\{\lambda\geq0:u\leq\psi_{\lambda}\ \mbox{in}\ \mathbb{R}%
^{n}\backslash B_{1}\}.
\]
By continuity, $u\leq\psi_{\lambda^{\ast}}$ in $\mathbb{R}^{n}\backslash
B_{1}$. If $\lambda^{\ast}=0$, then $0\leq u\leq\max_{\partial B_{1}}u$ in
$\mathbb{R}^{n}\backslash B_{1}$. By Theorem 5.4, $u$ has a limit at infinity.
Now we assume $\lambda^{\ast}>0$ and our aim is to show that also $u\geq
\phi_{\lambda^{\ast}}$ in $\mathbb{R}^{n}\backslash B_{1}$.

For all positive integers $k>\max\{10,\frac{2}{\lambda^{\ast}}\}$, there exist
$y^{k}$ such that $|y^{k}|\geq e^{k^{2}}$, $|y^{k+1}|>|y^{k}|$ and
$u(y^{k})>\psi_{\lambda^{\ast}-\frac{1}{k}}(y^{k})$. By
(\ref{SolRadialAsymp2d}), there exists $\hat{k}$ such that for all $k\geq
\hat{k}$, we have $\psi_{\lambda^{\ast}}(y^{k})-u(y^{k})<\psi_{\lambda^{\ast}%
}(y^{k})-\psi_{\lambda^{\ast}-\frac{1}{k}}(y^{k})<\frac{2}{k}\ln|y^{k}|$. The
function $w(x):=\psi_{\lambda^{\ast}}(x)-u(x)$ satisfies equation
(\ref{Euniform}) with
\begin{equation}
a_{ij}(x)=\int_{0}^{1}\frac{\delta_{ij}}{\sqrt{1-|Dw^{t}|^{2}}}+\frac
{w_{i}^{t}w_{j}^{t}}{(\sqrt{1-|Dw^{t}|^{2}})^{3}}dt \label{Coefficient}%
\end{equation}
where $w^{t}:=(1-t)u+t\psi_{\lambda^{\ast}}$. By Theorem 5.3, we have
$\psi_{\lambda^{\ast}}(x)-u(x)<\frac{2\cdot\Gamma}{k}\ln|x|$ on $\partial
B_{|y^{k}|}$ for all $k\geq\hat{k}$. Fix any small $\delta>0$. Note that
$\psi_{\lambda^{\ast}}(x)-\phi_{\lambda^{\ast}-\delta}(x)>\frac{\delta}{2}%
\ln|x|$ outside some ball. So there exist $\tilde{k}$ such that $u(x)>\phi
_{\lambda^{\ast}-\delta}(x)$ on $\partial B_{|y^{k}|}$ for all $k\geq\tilde
{k}$. Thus $u\geq\phi_{\lambda^{\ast}-\delta}$ in $\mathbb{R}^{n}\backslash
B_{1}$ by comparison principle. By continuity, we have $u\geq\phi
_{\lambda^{\ast}}$ in $\mathbb{R}^{n}\backslash B_{1}$.

Now we have established that $\phi_{\lambda^{\ast}}\leq u\leq\psi
_{\lambda^{\ast}}(x)$ in $\mathbb{R}^{n}\backslash B_{1}$. That is $0\leq
\psi_{\lambda^{\ast}}(x)-u\leq\max_{\partial B_{1}}u$. So by Theorem 5.4,
$\psi_{\lambda^{\ast}}(x)-u$ has a limit at infinity. Denote this
$\lambda^{\ast}=d$, then we have
\[
u(x)=c+d\ln|x|+o(1)
\]
as $|x|\rightarrow\infty$ for some constant $c$. Since we assumed $u$ is
bounded below, the constant $d\geq0$. If $u$ is bounded above, then we have
$u(x)=c+d\ln|x|$+o(1) with $d\leq0$.

\vspace{0.618ex} \noindent\textit{Step 3.} (Improve $o(1)$ to $O(|x|^{-1})$.)

We still assume $u\geq0$ as above. Suppose $d>0$. Choose $R_{0}>10$ such that
$|Du(x)|<\frac{1}{10}$ and $u(x)<2d\ln|x|$ when $|x|\geq R_{0}$. For any point
$x$ with $|x|:=2R\geq2R_{0}$, define $v(y):=\frac{u(Ry+x)}{R}$. Since $u$
satisfies the non-divergence form equation (\ref{Endiv}), $v(y)$ satisfies the
equation $a_{ij}(y)v_{ij}(y)=0$ for $y\in B_{1}$ with $a_{ij}(y)=\delta
_{ij}+\frac{v_{i}v_{j}}{1-|Dv|^{2}}$. By Morrey-Nirenberg's $C^{1,\alpha}$
estimate for 2 dimensional uniformly elliptic non-divergence form equation
[GT98,Theorem 12.4], for some $\alpha>0$ we have
\begin{equation}
\Vert v\Vert_{C^{1,\alpha}(B_{\frac{1}{2}})}\leq C\Vert v\Vert_{L^{\infty
}(B_{1})}\leq\frac{C\ln|x|}{|x|} \label{EstLnv}%
\end{equation}
where $C$ is a universal constant. In particular, it means that
\begin{equation}
|Du(x)|=|Dv(0)|\leq\frac{C\ln|x|}{|x|},\ \mbox{for}\ |x|\geq2R_{0}.
\label{EstLnu}%
\end{equation}

Let $e$ be any unit vector, then $v_{e}$ satisfies the equation $(a_{ij}%
(y)(v_{e})_{j})_{i}=0$ in $B_{1}$, with $a_{ij}=\frac{\delta_{ij}}%
{\sqrt{1-|Dv|^{2}}}+\frac{v_{i}v_{j}}{(\sqrt{1-|Dv|^{2}})^{3}}$. By
(\ref{EstLnv}), $\Vert a_{ij}\Vert_{C^{\alpha}(B_{\frac{1}{2}})}$ is bounded
by a universal constant. By Schauder estimate [GT98, Theorem 8.32],
\[
|Dv_{e}(0)|\leq C\Vert v_{e}\Vert_{L^{\infty}(B_{\frac{1}{2}})}\leq\frac
{C\ln|x|}{|x|}.
\]
Note that $Ru_{ee}(x)=v_{ee}(0)$, so we have
\begin{equation}
|D^{2}u(x)|\leq\frac{C\ln|x|}{|x|^{2}},\ \mbox{for}\ |x|\geq2R_{0}.
\label{EstLn2}%
\end{equation}
In fact, using bootstrap argument, we have
\begin{equation}
|D^{k}u(x)|\leq\frac{C\ln|x|}{|x|^{k}},\ \mbox{for}\ |x|\geq2R_{0},
\label{EstLnk}%
\end{equation}
for all $k=1,2,\cdots$.

We write equation (\ref{Endiv}) as
\[
\triangle u=\frac{-(Du)^{^{\prime}}D^{2}uDu}{1-|Du|^{2}}%
:=f(x),\ \mbox{in}\ \mathbb{R}^{n}\backslash B_{2R_{0}}.
\]
Then $|f(x)|\leq\frac{C(\ln|x|)^{3}}{|x|^{4}}$ by (\ref{EstLnu}) and
(\ref{EstLn2}). Define $K[u](x):=u(\frac{x}{|x|^{2}})$ for $x\in B_{\frac
{1}{2R_{0}}}\backslash\{0\}$. Then
\[
\triangle K[u]=|x|^{-4}f(\frac{x}{|x|^{2}}):=g(x),\ \mbox{in}\ B_{\frac
{1}{2R_{0}}}\backslash\{0\}
\]
with $|g(x)|\leq C(-\ln|x|)^{3}$. Let $N[g]$ be the Newtonian potential of $g$
in $B_{\frac{1}{2R_{0}}}$. Since $g$ is in $L^{p}(B_{\frac{1}{2R_{0}}})$ for
any $p>0$, $N[g]$ is in $W^{2,p}$ for any $p$ and hence is in $C^{1,\alpha}$
for any $0<\alpha<1$. Now $K[u]-N[g]$ is harmonic in $B_{\frac{1}{2R_{0}}%
}\backslash\{0\}$. Notice that $|K[u](x)|\leq-2d\ln|x|+C$ in $B_{\frac
{1}{2R_{0}}}\backslash\{0\}$, so $|K[u]-N[g]|\leq-2d\ln|x|+C$ in $B_{\frac
{1}{2R_{0}}}\backslash\{0\}$. Therefore $K[u]-N[g]$ is the sum of $c_{1}%
\ln|x|$ (for some constant $c_{1}$) and a harmonic function in $B_{\frac
{1}{2R_{0}}}$. So $K[u](x)$ is the sum of $c_{1}\ln|x|$ and a $C^{1,\alpha}$
function in $B_{\frac{1}{2R_{0}}}$. Fix an $\alpha\in(0,1)$, for some affine
function $c_{2}+b\cdot x$, we have $|K[u](x)-(c_{1}\ln|x|+c_{2}+b\cdot x)|\leq
C|x|^{1+\alpha}$ in $B_{\frac{1}{2R_{0}}}\backslash\{0\}$. Go back to $u$ and
we have $|u(x)-(-c_{1}\ln|x|+c_{2}+b\cdot\frac{x}{|x|^{2}})|\leq
C|x|^{-1-\alpha}$ for $|x|\geq2R_{0}$. From the result of Step 2, we must have
$-c_{1}=d$ and $c_{2}=c$. Thus
\[
u(x)=c+d\ln|x|+O(|x|^{-1}).
\]

\vspace{0.618ex} \noindent\textit{Step 4.} (Improve $O(|x|^{-1})$ to
$O_{k}(|x|^{-1})$.)

Since $\psi_{d}(x)=\tilde{c}+d\ln|x|+O_{k}(|x|^{-1})$ for some $\tilde{c}$, we
consider $w(x):=\psi_{d}(x)-u(x)-\tilde{c}+c=O(|x|^{-1})$. The function $w$
satisfies the equation $(a_{ij}w_{j})_{i}=0$ with $a_{ij}$ given by
(\ref{Coefficient}). In view of (\ref{EstLnk}) and $|D^{k}\psi_{d}%
(x)|\leq\frac{C}{|x|^{k}}$, we have $|D^{k}w^{t}(x)|\leq\frac{C\ln|x|}%
{|x|^{k}}$ and hence $|D^{k}a_{ij}(x)|\leq\frac{C(k)(\ln|x|)^{k}}{|x|^{k+2}}$.
For any point $x$ with $|x|:=4R\geq4R_{0}$, define $v(y):=\frac{w(Ry+x)}{R}$.
The $v$ satisfies the equation $(\tilde{a}_{ij}v_{j})_{i}=0$ with $\tilde
{a}_{ij}(y)=a_{ij}(x+Ry)$. We have $\Vert D^{k}\tilde{a}_{ij}\Vert
_{C^{0}(B_{1})}=R^{k}\Vert D^{k}a_{ij}\Vert_{C^{0}(B_{R}(x))}\leq C(k)$ and so
$\Vert\tilde{a}_{ij}\Vert_{C^{k}(B_{1})}\leq C(k)$ for all $k$. Then by
Schauder estimate,
\[
R^{k-1}|D^{k}w(x)|=|D^{k}v(0)|\leq C(k)\Vert v\Vert_{L^{\infty}(B_{1})}%
\leq\frac{C(k)}{|x|^{2}},
\]
and hence $|D^{k}w(x)|\leq\frac{C(k)}{|x|^{k+1}}$ for $|x|\geq4R_{0}$. This
means $w(x)=O_{k}(|x|^{-1})$ and hence
\[
u(x)=c+d\ln|x|+O_{k}(|x|^{-1}).
\]

\vspace{0.618ex} \noindent\textit{Step 5.} (Ascertain the value of $d$.) \ %

\begin{align*}
Res[u]  &  =\frac{1}{2\pi}\int_{\partial B_{r}}\frac{\partial u/\partial
\vec{n}}{\sqrt{1-|Du|^{2}}}ds\\
&  =\frac{1}{2\pi}\int_{0}^{2\pi}(\frac{d}{r}+O(r^{-2}))r d\theta=d+O(r^{-1}).
\end{align*}
Letting $r\rightarrow\infty$, we have $d=Res[u]$.

\subsection{Case $a=0$, $n\geq3$}

\ 

\vspace{0.618ex} \noindent\textit{Step 1.} ($|u(x)|\leq c$ for large $c$.)

We still assume $u\geq0$ and define $\phi_{\lambda}$ and $\psi_{\lambda}$ as
above. Using the same method, we can prove $u(x)\leq\psi_{\lambda}(x)$ for
some large $\lambda$ in $\mathbb{R}^{n}\backslash B_{1}$. But in the
dimensions $n\geq3$, $\psi_{\lambda}$ is bounded.

\vspace{0.618ex} \noindent\textit{Step 2.} ($u(x)=u_{\infty}+O(|x|^{2-n})$.)

Since $u$ is bounded, applying Theorem 5.4 directly to $u$, we have
$u(x)=u_{\infty}+o(1)$ where $u_{\infty}:=\lim_{x\rightarrow\infty}u(x)$.
Define $\phi_{\lambda}(x):=w_{\lambda}(x)-w_{\lambda}(e_{1})+\min_{\partial
B_{1}}u$ and $\psi_{\lambda}(x):=w_{\lambda}(x)-w_{\lambda_{1}}(e_{1}%
)+\max_{\partial B_{1}}u$ for $\lambda\in(-\infty,+\infty)$. We can choose
$\lambda_{1}$ and $\lambda_{2}$ such that
\[
\lim_{x\rightarrow\infty}\phi_{\lambda_{1}}(x)=\lim_{x\rightarrow\infty}%
\psi_{\lambda_{2}}(x)=u_{\infty}.
\]
By comparison principle,
\[
\phi_{\lambda_{1}}(x)\leq u(x)\leq\psi_{\lambda_{2}}(x),
\]
and this means that
\[
u(x)=u_{\infty}+O(|x|^{2-n}).
\]

\vspace{0.618ex} \noindent\textit{Step 3.} ($u(x)=u_{\infty}-d|x|^{2-n}%
+O(|x|^{1-n})$ for some $d$.)

We adopt the same strategy as in the step 3 of above subsection: establish the
decay rate of $|Du(x)|$ and $|D^{2}u(x)|$, make Kelvin transform to
$u(x)-u_{\infty}$ and estimate the Newtonian potential of right hand side. The
only difference is that: when we estimate the decay rate of $|Du(x)|$ we can
not use Morrey's $C^{1,\alpha}$ estimate which is only true for 2 dimension,
alternatively the first order derivatives of $u$ (so is $v(y):=\frac
{u(x+Ry)-u_{\infty}}{R}$) satisfy a uniformly elliptic divergence form
equation and thus we can apply De Giorgi-Nash's Theorem [GT98, Chapter 8] to
$Dv.$
\begin{align*}
\Vert Dv\Vert_{C^{\alpha}(B_{\frac{1}{2}})}  &  \leq C\Vert Dv\Vert
_{L^{2}(B_{\frac{3}{4}})}\\
&  \leq C\Vert v\Vert_{L^{2}(B_{1})}\ \ \mbox{(Caccioppoli)}\\
&  \leq C\Vert v\Vert_{L^{\infty}(B_{1})}\leq C|x|^{1-n}.
\end{align*}
This treatment also fits two dimensional case certainly. We leave the
remaining details to the readers.

\vspace{0.618ex} \noindent\textit{Step 4.} (Improve $O(|x|^{1-n})$ to
$O_{k}(|x|^{1-n})$.)

Do the same thing to $u(x)-u_{\infty}$ as in the step 4 of above subsection.

\vspace{0.618ex} \noindent\textit{Step 5.} (Ascertain the value of $d$.) \ %

\begin{align*}
Res[u]  &  =\frac{1}{(n-2)|\partial B_{1}|}\int_{\partial B_{r}}\frac{\partial
u/\partial\vec{n}}{\sqrt{1-|Du|^{2}}}d\sigma\\
&  =\frac{1}{(n-2)|\partial B_{1}|}\int_{\partial B_{1}}(\frac{(n-2)d}%
{r^{n-1}}+O(r^{-n}))r^{n-1} dS^{n-1}=d+O(r^{-1}).
\end{align*}
Letting $r\rightarrow\infty$, we have $d=Res[u]$.

\subsection{Case $|a|>0$\textbf{, }$n=2$}

\ \vspace{0.618ex}

By a rotation, we can assume $a=(0,\eta)$ with $\eta\in(0,1)$. Make the
Lorentz transformation $L_{-\eta}:\mathbb{L}^{2+1}\rightarrow\mathbb{L}^{2+1}%
$,
\[
L_{-\eta}:(x_{1},x_{2},t)\rightarrow(x_{1},\frac{x_{2}-\eta t}{\sqrt
{1-\eta^{2}}},\frac{-\eta x_{2}+t}{\sqrt{1-\eta^{2}}}):=(\tilde{x}_{1}%
,\tilde{x}_{2},\tilde{t}).
\]
Then the plane $\{t=\eta x_{2}\}$ was transformed to the plane $\{\tilde
{t}=0\}$ and the graph of $u$ over $\mathbb{R}^{2}\backslash A$ was
transformed to another maximal hypersurface which is the graph of some
function (say $\tilde{u}$) defined on $\mathbb{R}^{2}\backslash\tilde{A}$ for
some bounded closed set $\tilde{A}$. The blowdown of $\tilde{u}$ is the $0$
function. So $\tilde{u}$ has the asymptotic expansion:
\begin{equation}
\tilde{u}(\tilde{x})=\tilde{c}+\tilde{d}\ln|\tilde{x}|+O(|\tilde{x}|^{-1}).
\label{EstLnnew}%
\end{equation}
Transform back and make some direct computations, we can establish the
asymptotic expansion of $u$. The details are as follows.

The Lorentz transformation
\begin{equation}
L_{\eta}:(\tilde{x}_{1},\tilde{x}_{2},\tilde{u}(\tilde{x}_{1},\tilde{x}%
_{2}))\rightarrow(\tilde{x}_{1},\frac{\tilde{x}_{2}+\eta\tilde{u}}%
{\sqrt{1-\eta^{2}}},\frac{\eta\tilde{x}_{2}+\tilde{u}}{\sqrt{1-\eta^{2}}%
})=(x_{1},x_{2},u(x_{1},x_{2})). \label{Lorentz}%
\end{equation}
Use the polar coordinates $x_{1}=r\cos\theta$, $x_{2}=r\sin\theta$ and
substitute (\ref{EstLnnew}) to (\ref{Lorentz}), we get
\begin{equation}
\tilde{x}_{2}+\eta\tilde{c}+\frac{\eta\tilde{d}}{2}\ln(r^{2}\cos^{2}%
\theta+\tilde{x}_{2}^{2})+O(\frac{1}{\sqrt{r^{2}\cos^{2}\theta+\tilde{x}%
_{2}^{2}}})=r\sin\theta\sqrt{1-\eta^{2}}. \label{LorentzPolar}%
\end{equation}
We want to solve $\tilde{x}$ from (\ref{LorentzPolar}) and substitute it to
(\ref{EstLnnew}) and the third equality of (\ref{Lorentz}), then we will get
the expansion of $u$. We need to solve $\tilde{x}$ three times iteratively.

Firstly, we assume $\sin\theta\neq0$. From (\ref{LorentzPolar}) we can see
\[
\tilde{x}_{2}=r\sin\theta\sqrt{1-\eta^{2}}(1+O(\frac{\ln r}{r}%
))\ \ \mbox{as}\ r\rightarrow+\infty.
\]
Then
\[
r^{2}\cos^{2}\theta+\tilde{x}_{2}^{2}=r^{2}(1-\eta^{2}\sin^{2}\theta
)(1+O(\frac{\ln r}{r})),
\]
and hence
\begin{equation}
\ln(r^{2}\cos^{2}\theta+\tilde{x}_{2}^{2})=2\ln(r\sqrt{1-\eta^{2}\sin
^{2}\theta})+O(\frac{\ln r}{r}) \label{Iteration2}%
\end{equation}
where $O(\ln r/r)$ is independent of small $\sin\theta$. Substitute
(\ref{Iteration2}) to (\ref{LorentzPolar}) and solve $\tilde{x}_{2}$ again,
\begin{equation}
\tilde{x}_{2}=r\sin\theta\sqrt{1-\eta^{2}}-\eta\tilde{c}-\eta\tilde{d}%
\ln(r\sqrt{1-\eta^{2}\sin^{2}\theta})+O(\frac{\ln r}{r}). \label{Iteration3}%
\end{equation}
Now we have
\[
r^{2}\cos^{2}\theta+\tilde{x}_{2}^{2}=r^{2}(1-\eta^{2}\sin^{2}\theta
)(1-\frac{2\eta\sqrt{1-\eta^{2}}\tilde{d}\sin\theta\ln r}{(1-\eta^{2}\sin
^{2}\theta)r}+O(\frac{1}{r})),
\]
and
\begin{equation}
\ln(r^{2}\cos^{2}\theta+\tilde{x}_{2}^{2})=2\ln(r\sqrt{1-\eta^{2}\sin
^{2}\theta})-\frac{\eta\sqrt{1-\eta^{2}}\tilde{d}\sin\theta}{(1-\eta^{2}%
\sin^{2}\theta)}\cdot\frac{\ln r}{r}+O(\frac{1}{r}). \label{Iteration4}%
\end{equation}
Substitute (\ref{Iteration4}) to (\ref{LorentzPolar}) and solve $\tilde{x}%
_{2}$ again,
\begin{align}
\tilde{x}_{2}  &  =r\sin\theta\sqrt{1-\eta^{2}}-\eta\tilde{c}-\eta\tilde{d}%
\ln(r\sqrt{1-\eta^{2}\sin^{2}\theta})\nonumber\\
&  +\frac{\eta\sqrt{1-\eta^{2}}\tilde{d}\sin\theta}{(1-\eta^{2}\sin^{2}%
\theta)}\cdot\frac{\ln r}{r}+O(\frac{1}{r}). \label{Iteration5}%
\end{align}
Substitute (\ref{Iteration4}) to (\ref{EstLnnew}) and then substitute
(\ref{EstLnnew}) and (\ref{Iteration5}) to the third equality of
(\ref{Lorentz}), we have
\begin{align}
u(r,\theta)  &  =\eta r\sin\theta+\sqrt{1-\eta^{2}}\tilde{c}+\sqrt{1-\eta^{2}%
}\tilde{d}\ln(r\sqrt{1-\eta^{2}\sin^{2}\theta})\nonumber\\
&  +\frac{\eta^{2}\tilde{d}\sin\theta}{(1-\eta^{2}\sin^{2}\theta)}\cdot
\frac{\ln r}{r}+O(\frac{1}{r}). \label{Asymptote2dOld}%
\end{align}

Notice that we get (\ref{Asymptote2dOld}) with the assumption $\sin\theta
\neq0$. If $\sin\theta=0$, then (\ref{LorentzPolar}) becomes
\begin{align*}
\tilde{x}_{2}+\eta\tilde{c} +\frac{\eta\tilde{d}}{2}\ln(r^{2}+\tilde{x}%
_{2}^{2})+O(\frac{1}{\sqrt{r^{2}+\tilde{x}_{2}^{2}}})=0.
\end{align*}
Then we have
\begin{align*}
\tilde{x}_{2}=-\eta\tilde{d}\ln r(1+o(1)),
\end{align*}
\begin{align*}
r^{2}+\tilde{x}_{2}^{2}=r^{2} (1+O(\frac{1}{r})),
\end{align*}
\begin{align*}
\ln(r^{2}+\tilde{x}_{2}^{2})=2\ln r+O(\frac{1}{r}),
\end{align*}
\begin{align*}
\tilde{x}_{2}=-\eta\tilde{d}\ln r-\eta\tilde{c}+O(\frac{1}{r}),
\end{align*}
and hence
\begin{align*}
u(r,\theta) = \sqrt{1-\eta^{2}}\tilde{c}+\sqrt{1-\eta^{2}}\tilde{d}\ln r
+O(\frac{1}{r}).
\end{align*}
This means (\ref{Asymptote2dOld}) is also true for $\sin\theta=0$.

Let $\sqrt{1-\eta^{2}}\tilde{c}:=c$ and $\sqrt{1-\eta^{2}}\tilde{d}:=d$. In
$x$ coordinates, we have
\begin{align*}
u(x_{1},x_{2})  &  =\eta x_{2}+c+d\ln\sqrt{x_{1}^{2}+(1-\eta^{2})x_{2}^{2}}\\
&  +\frac{\eta^{2}d|x|x_{2}}{\sqrt{1-\eta^{2}}(x_{1}^{2}+(1-\eta^{2})x_{2}%
^{2})}\cdot\frac{\ln|x|}{|x|}+O(|x|^{-1}).
\end{align*}
Getting rid of the assumption $a=(0,\eta)$, it is not hard to see that
\begin{align*}
u(x)  &  =a\cdot x+c+d\ln\sqrt{|x|^{2}-(a\cdot x)^{2}}\\
&  +\frac{d|a||x|(a\cdot x)}{\sqrt{1-|a|^{2}}(|x|^{2}-(a\cdot x)^{2})}%
\cdot\frac{\ln|x|}{|x|}+O(|x|^{-1}).
\end{align*}
By the method in Step 4 of Section 6.1, we can improve $O(|x|^{-1})$ to
$O_{k}(|x|^{-1})$. We omit the details.

The remaining task is to compute $d$ in terms of $Res[u]$ and $|a|$. For
simplicity, we still assume $a=(0,\eta)$. Consider the ellipse
\[
E_{\rho}:=\{x_{1}^{2}+(1-\eta^{2})x_{2}^{2}=\rho^{2}\}.
\]
Use the polar coordinates, but this time we set $x_{1}=r\cos\theta$,
$\sqrt{1-\eta^{2}}x_{2}=r\sin\theta$. So $E_{\rho}=\{(r,\theta): r=\rho,
0\leq\theta<2\pi\}$. On $E_{\rho}$:%

\begin{align*}
Du(\theta)=(\frac{d\cos\theta}{\rho}+o(\rho^{-1}),\eta+\frac{d\sqrt{1-\eta
^{2}}\sin\theta}{\rho}+o(\rho^{-1})),
\end{align*}
the unit outward normal vector
\begin{align*}
\vec{n}(\theta)=(\frac{\cos\theta}{\sqrt{1-\eta^{2}\sin^{2}\theta}}%
,\frac{\sqrt{1-\eta^{2}}\sin\theta}{\sqrt{1-\eta^{2}\sin^{2}\theta}})
\end{align*}
and the length element
\begin{align*}
ds=\frac{\sqrt{1-\eta^{2}\sin^{2}\theta}}{\sqrt{1-\eta^{2}}} \rho d\theta.
\end{align*}
So
\begin{align*}
\frac{\partial u/\partial\vec{n}}{\sqrt{1-|Du|^{2}}}=\frac{\eta\sin\theta
}{\sqrt{1-\eta^{2}\sin^{2}\theta}}+\frac{d}{\rho\sqrt{1-\eta^{2}}\sqrt
{1-\eta^{2}\sin^{2}\theta}}+o(\rho^{-1})
\end{align*}
and hence
\begin{align*}
Res[u]  &  =\frac{1}{2\pi}\int_{E_{\rho}}\frac{\partial u/\partial\vec{n}%
}{\sqrt{1-|Du|^{2}}}ds\\
&  =\frac{1}{2\pi}\int_{0}^{2\pi}\frac{\eta\rho\sin\theta}{\sqrt{1-\eta^{2}}%
}d\theta+\frac{1}{2\pi}\int_{0}^{2\pi}\frac{d}{1-\eta^{2}}d\theta+o(1)\\
&  =\frac{d}{1-\eta^{2}}+o(1).
\end{align*}
Letting $\rho\rightarrow+\infty$, we have
\[
d=(1-\eta^{2})Res[u]=(1-|a|^{2})Res[u].
\]

\subsection{Case $|a|>0$\textbf{, }$n\geq3$}

\ \vspace{0.618ex}

We do the same things as above. Assuming $a=(0,\eta)$, make the Lorentz
transformation $L_{-\eta}$: graph of $u\rightarrow$ graph of $\tilde{u}$,
then
\begin{align*}
\tilde{u}(\tilde{x})=\tilde{c}-\tilde{d}|\tilde{x}|^{2-n}+O(|\tilde{x}|^{1-n})
\end{align*}
and
\begin{align*}
L_{\eta}: (\tilde{x}^{\prime},\tilde{x}_{n},\tilde{u}(\tilde{x}^{\prime
},\tilde{x}_{n}))\rightarrow(\tilde{x}^{\prime},\frac{\tilde{x}_{n}+\eta
\tilde{u}}{\sqrt{1-\eta^{2}}},\frac{\eta\tilde{x}_{n}+ \tilde{u}}{\sqrt
{1-\eta^{2}}})=(x^{\prime},x_{n},u(x^{\prime},x_{n})).
\end{align*}
Use the polar coordinates $x^{\prime}=r\cos\theta\xi$, $x_{n}=r\sin\theta$
with $-\frac{\pi}{2}\leq\theta\leq\frac{\pi}{2}$ and $\xi\in S^{n-2}$ the unit
sphere in $\mathbb{R}^{n-1}$. Then we are going to solve $\tilde{x}_{n}$ from
\begin{align*}
\tilde{x}_{n}+\eta\tilde{c} -\eta\tilde{d}(r^{2}\cos^{2}\theta+\tilde{x}%
_{n}^{2})^{\frac{2-n}{2}}+O((r^{2}\cos^{2}\theta+\tilde{x}_{n}^{2}%
)^{\frac{1-n}{2}})=r\sin\theta\sqrt{1-\eta^{2}}.
\end{align*}
Suppose $\sin\theta\neq0$. We have
\begin{align*}
\tilde{x}_{n}=r\sin\theta\sqrt{1-\eta^{2}}(1+O(\frac{1}{r})),
\end{align*}
\begin{align*}
r^{2}\cos^{2}\theta+\tilde{x}_{n}^{2}=r^{2}(1-\eta^{2}\sin^{2}\theta
)(1+O(\frac{1}{r})),
\end{align*}
\begin{align*}
(r^{2}\cos^{2}\theta+\tilde{x}_{n}^{2})^{\frac{2-n}{2}}=r^{2-n}(1-\eta^{2}%
\sin^{2}\theta)^{\frac{2-n}{2}}+O(r^{1-n})
\end{align*}
where $O(r^{1-n})$ is independent of small $\sin\theta$. So
\begin{align*}
\tilde{u}=\tilde{c}-\tilde{d}r^{2-n}(1-\eta^{2}\sin^{2}\theta)^{\frac{2-n}{2}%
}+O(r^{1-n}),
\end{align*}
and
\begin{align*}
\tilde{x}_{n}=r\sin\theta\sqrt{1-\eta^{2}} -\eta\tilde{c} +\eta\tilde{d}
r^{2-n}(1-\eta^{2}\sin^{2}\theta)^{\frac{2-n}{2}}+O(r^{1-n}).
\end{align*}
Therefore, denoting $\sqrt{1-\eta^{2}}\tilde{c}:=c$ and $\sqrt{1-\eta^{2}%
}\tilde{d}:=d$,
\begin{align*}
u(x)  &  = \eta x_{n}+c-d(|x|^{2}-\eta^{2}x_{n}^{2})^{\frac{2-n}{2}%
}+O(|x|^{1-n})\\
&  = a\cdot x+c-d(|x|^{2}-(a\cdot x)^{2})^{\frac{2-n}{2}}+O(|x|^{1-n}).
\end{align*}
One can verify that the above expansion is also true in the case of
$\sin\theta=0$. Also $O(|x|^{1-n})$ can be improved to $O_{k}(|x|^{1-n})$. We
omit the details.

Now we compute $d$. Assume $a=(0,\eta)$ and
\[
E_{\rho}:=\{|x^{\prime2}+(1-\eta^{2})x_{n}^{2}=\rho^{2}\}.
\]
Use the coordinates: $x^{\prime}=r\cos\theta\xi$, $\sqrt{1-\eta^{2}}%
x_{n}=r\sin\theta$. So
\[
E_{\rho}=\{(r,\theta,\xi): r=\rho, -\frac{\pi}{2}\leq\theta\leq\frac{\pi}{2},
\xi\in S^{n-2}\}.
\]
On $E_{\rho}$:%

\begin{align*}
u_{i}=\frac{(n-2)d x_{i}}{r^{n}}+O(r^{-n})\ \ \mbox{for}\ i=1,\cdots,n-1
\end{align*}
and
\begin{align*}
u_{n}=\eta+\frac{(n-2)d (1-\eta^{2})x_{n}}{r^{n}}+O(r^{-n}).
\end{align*}
The unit outward normal vector
\begin{align*}
\vec{n}=(\frac{x_{1}}{r\sqrt{1-\eta^{2}\sin^{2}\theta}},\cdots,\frac{x_{n-1}%
}{r\sqrt{1-\eta^{2}\sin^{2}\theta}},\frac{(1-\eta^{2})x_{n}}{r\sqrt{1-\eta
^{2}\sin^{2}\theta}}).
\end{align*}
The surface element
\begin{align*}
d\sigma=\frac{\sqrt{1-\eta^{2}\sin^{2}\theta}}{\sqrt{1-\eta^{2}}} \rho
^{n-1}\cos^{n-2}\theta d\theta dS^{n-2}.
\end{align*}
So
\begin{align*}
\frac{\partial u/\partial\vec{n}}{\sqrt{1-|Du|^{2}}}=\frac{\eta\sin\theta
}{\sqrt{1-\eta^{2}\sin^{2}\theta}}+\frac{(n-2)d\rho^{1-n}}{\sqrt{1-\eta^{2}%
}\sqrt{1-\eta^{2}\sin^{2}\theta}}+O(\rho^{-n}),
\end{align*}
and hence
\begin{align*}
&  Res[u]=\frac{1}{(n-2)|\partial B_{1}|}\int_{E_{\rho}}\frac{\partial
u/\partial\vec{n}}{\sqrt{1-|Du|^{2}}}d\sigma\\
&  =\frac{|S^{n-2}|}{(n-2)|\partial B_{1}|}\int_{-\frac{\pi}{2}}^{\frac{\pi
}{2}}\frac{\eta\rho^{n-1}\cos^{n-2}\theta\sin\theta}{\sqrt{1-\eta^{2}}}%
+\frac{(n-2)d\cos^{n-2}\theta}{1-\eta^{2}}+O(\rho^{-1})d\theta\\
&  =\frac{d}{1-\eta^{2}}+O(\rho^{-1}).
\end{align*}
We used the fact that
\[
\int_{-\frac{\pi}{2}}^{\frac{\pi}{2}}|S^{n-2}|\cos^{n-2}\theta d\theta
=|S^{n-1}|=|\partial B_{1}|.
\]
Letting $\rho\rightarrow+\infty$, we have
\[
d=(1-\eta^{2})Res[u]=(1-|a|^{2})Res[u].
\]

\section{Exterior Dirichlet problem: proof of Theorem 1.2}

Recall $w_{\lambda}$ is the radially solution defined by (\ref{SolRadial}).
Let $a\in B_{1}$, we use $w_{\lambda}^{a}(x)$ to denote the representation
function of the hypersurface $L_{a}$(graph of $w_{\lambda}$), where the
Lorentz transformation $L_{a}=T_{a}L_{\left\vert a\right\vert }T_{a}^{-1}$ is
defined in the end of Section 2. Then the function $w_{\lambda}^{a}(x)$ has
the following properties: $w_{\lambda}^{a}(0)=0$, $w_{\lambda}^{a}(x)$ solves
equation (1) in $\mathbb{R}^{n}\backslash\{0\}$ and (from the argument in the
previous section or by direct calculation) for $n=2$
\[
w_{\lambda}^{a}(x)=a\cdot x+\sqrt{1-|a|^{2}}m(\lambda)+\sqrt{1-|a|^{2}}%
\lambda\ln\sqrt{|x|^{2}-(a\cdot x)^{2}}+o(1)
\]
and for $n\geq3$
\[
w_{\lambda}^{a}(x)=a\cdot x+\sqrt{1-|a|^{2}}M(\lambda,n)-\frac{\sqrt
{1-|a|^{2}}\lambda}{n-2}(|x|^{2}-(a\cdot x)^{2})^{\frac{2-n}{2}}+o(|x|^{2-n})
\]
as $x\rightarrow\infty$. The numbers $m(\lambda)$ and $M(\lambda,n)$ are from
(\ref{SolRadialAsymp2d}) and (\ref{SolRadialAsymp3d}).

Now we prove Theorem 1.2. We do this in the following two subsections
corresponding to the cases $n=2$ and $n\geq3$ respectively.

\subsection{Case $n=2$}

\ 

\vspace{0.618ex} Let $A$, $g$, $a$ and $d$ be given as in Theorem 1.2 and
$\sqrt{1-|a|^{2}}\lambda=d$. Choose constants $c^{-}\leq0\leq c^{+}$ such that
$w_{\lambda}^{a}(x)+c^{-}\leq g(x)\leq w_{\lambda}^{a}(x)+c^{+}$ on $\partial
A$. We claim that there exists $\breve{R}(A,g,a,d)>0$ such that for any
$R\geq\breve{R}$ there exists a solution $u_{R}$ of maximal surface equation
in $B_{R}\backslash A$ satisfying $u_{R}=g$ on $\partial A$ and $u_{R}%
=w_{\lambda}^{a}$ on $\partial B_{R}$.

In fact, let $\psi$ be an spacelike extension of $g$ into $\mathbb{R}%
^{2}\backslash A$. By Theorem 3.1, there exists $R^{\ast}$ such that
$|\psi(x)-\psi(y)|<|x-y|$ for any $x,y\in\partial B_{R^{\ast}}$ and $x\neq y$.
Assume $|\psi|\leq G$ on $\partial B_{R^{\ast}}$. Let $R\geq\breve{R}>R^{\ast
}$, for any $x\in\partial B_{R^{\ast}}$ and $y\in\partial B_{R}$,
\[
|w_{\lambda}^{a}(y)-\psi(x)|\leq|w_{\lambda}^{a}(y)|+G<\frac{|a|+1}%
{2}(R-R^{\ast})\leq\frac{|a|+1}{2}|x-y|
\]
provided $\breve{R}$ is chosen to be sufficiently large. So we can find a
spacelike function $v_{R}$ on $\partial B_{R}\backslash B_{R^{\ast}}$ such
that $v_{R}=\psi$ on $\partial B_{R^{\ast}}$ and $v_{R}=w_{\lambda}^{a}$ on
$\partial B_{R}$. Define $\Psi_{R}$ by $\Psi_{R}=\psi$ in $B_{R^{\ast}%
}\backslash A$ and $\Psi_{R}=v_{R}$ in $B_{R}\backslash B_{R^{\ast}}$. It is
not difficult to see that $\Psi_{R}$ is a spacelike function defined on
$B_{R}\backslash A$ possessing boundary values $g$ and $w_{\lambda}^{a}$ on
$\partial A$ and $\partial B_{R}$ respectively. Hence by Remark 2.2, we can
get $u_{R}$ by solving the Dirichlet problem.\ The above claim is proved.

By comparison principle, $w_{\lambda}^{a}(x)+c^{-}\leq u_{R}(x)\leq
w_{\lambda}^{a}(x)+c^{+}$ in $B_{R}\backslash A$. Choose any sequence of
$\breve{R}<R_{j}\rightarrow\infty$, by compactness, there exists a subsequence
of $\{u_{R_{j}}\}$ converging to a function $u$ locally uniformly in
$\mathbb{R}^{n}\backslash A$. By Lemma 2.1, $u$ is area maximizing. If $u$ is
not maximal, then graph $u$ contains a segment of light ray and hence the
whole of the ray in $(\mathbb{R}^{n}\backslash A)\times\mathbb{R}$
contradicting the fact $w_{\lambda}^{a}(x)+c^{-}\leq u(x)\leq w_{\lambda}%
^{a}(x)+c^{+}$. Therefore $u$ solves equation (\ref{Ediv}) in $\mathbb{R}%
^{n}\backslash A$. Moreover, $u=g$ on $\partial A$ and
\[
u(x)=a\cdot x+d\ln\sqrt{|x|^{2}-(a\cdot x)^{2}}+O(1)
\]
as $x\rightarrow\infty$.

Finally, we prove the uniqueness of $u$. Suppose there is another such
solution $v$ also satisfying $v=g$ on $\partial A$ and
\[
v(x)=a\cdot x+d\ln\sqrt{|x|^{2}-(a\cdot x)^{2}}+O(1)
\]
Then $w:=u-v$ satisfies a divergence form elliptic equation in $\mathbb{R}%
^{n}\backslash A$, $w=0$ on $\partial A$ and $w$ is bounded. By [GS56, Theorem
7], $w\equiv0$ in $\mathbb{R}^{n}\backslash A$.

\subsection{Case $n\geq3$}

\ 

\vspace{0.618ex} Given $A$, $g$, $a$ and $c$ as in Theorem 1.2, choose
$\tilde{R}$ and $G$ such that $A\in B_{\tilde{R}}$ and $|g|\leq G$ on
$\partial A$. Choose $\lambda^{\ast}>0$ such that $\sqrt{1-|a|^{2}}%
M(\lambda^{\ast},n)\geq|c|+\tilde{R}+G$. Denote
\begin{align*}
&  \Psi^{-}(x):=w_{\lambda^{\ast}}^{a}(x)-\sqrt{1-|a|^{2}}M(\lambda^{\ast
},n)+c,\\
&  \Psi^{+}(x):=w_{-\lambda^{\ast}}^{a}(x)+\sqrt{1-|a|^{2}}M(\lambda^{\ast
},n)+c.
\end{align*}
One can verify that
\begin{align*}
&  \Psi^{-}(x)\leq a\cdot x+c\leq\Psi^{+}(x)\ \ \ \mbox{in}\ \mathbb{R}^{n},\\
&  \Psi^{\pm}(x)=a\cdot x+c+o(1)\ \ \ \mbox{as}\ x\rightarrow\infty,\\
&  \Psi^{-}(x)\leq g\leq\Psi^{+}(x)\ \ \ \mbox{on}\ \partial A.
\end{align*}

For the same reason as in the two dimensional case in the previous subsection,
there exists $\breve{R}$ such that for any $R\geq\breve{R}$ there exists a
solution $u_{R}$ in $B_{R}\backslash A$ satisfying $u_{R}=g$ on $\partial A$
and $u_{R}=a\cdot x +c$ on $\partial B_{R}$. And hence $\Psi^{-}(x)\leq u_{R}
\leq\Psi^{+}(x)$ in $B_{R}\backslash A$. In the same way, we can construct a
solution $u$ in $\mathbb{R}^{n}\backslash A$ satisfying $u=g$ on $\partial A$
and
\begin{align*}
u(x)=a\cdot x +c+o(1)
\end{align*}
as $x\rightarrow\infty$.

The uniqueness of $u$ follows from the comparison principle directly.

\section*{Acknowledgements}

Part of this paper was completed during the first author's visit to University
of Washington (Seattle). His visit was funded by China Scholarship Council.
The second author is partially supported by an NSF grant.

\end{document}